\documentclass{article}
\usepackage[dvips]{graphicx}
\usepackage{a4wide}

\usepackage[dvips]{graphics,color}      
\usepackage{verbatim}
\usepackage{amsmath}
\usepackage{xspace}
\usepackage{pifont}
\usepackage{graphicx}
\usepackage{amssymb}
\usepackage{epic, eepic}
\usepackage{dsfont}
\usepackage{amssymb}
\usepackage{makeidx}

\usepackage{amsmath,afterpage}
\usepackage{epsf}

\def\0{\mathbf{0}}

\def\eps{\varepsilon}

\def\rr{\rightarrow}
\def\dr{\downarrow}

\def\beqa{\begin{eqnarray}}
\def\eeqa{\end{eqnarray}}
\def\beqas{\begin{eqnarray*}}
\def\eeqas{\end{eqnarray*}}

\newtheorem{theorem}{Theorem}[section]
\newtheorem{lemma}[theorem]{Lemma}
\newtheorem{proposition}[theorem]{Proposition}

\newtheorem{corollary}[theorem]{Corollary}

\newtheorem{remark}[theorem]{Remark}

\newtheorem{definition}[theorem]{Definition}

\numberwithin{equation}{section}
\newcommand{\old}[1]{{}}

\def\endpf{{\ \hfill\hbox{\vrule width1.0ex height1.0ex}\parfillskip 0pt}}
\newenvironment{proof}{\noindent{\bf Proof:}}{\endpf}

\addtocounter{footnote}{0}

\newcommand{\qed}{\hfill\rule{2mm}{2mm}}

\newcommand{\bd}{\begin{displaymath}}
\newcommand{\ed}{\end{displaymath}}
\newcommand{\be}{\begin{equation}}
\newcommand{\ee}{\end{equation}}
\newcommand{\bq}{\begin{eqnarray}}
\newcommand{\eq}{\end{eqnarray}}
\newcommand{\bn}{\begin{eqnarray*}}
\newcommand{\en}{\end{eqnarray*}}
\newcommand{\dl}{\delta}
\newcommand{\re}{\mathds{R}}

\title{The Multifractal Nature of Volterra-L\'{e}vy Processes}
\author{Eyal Neuman   \\ \\ Faculty of Industrial Engineering \\ and Management  \\ Technion - Institute of Technology \\ Haifa 3200 \\ Israel }
\date{}

\begin{document}

\maketitle

\paragraph{Abstract.}
We consider the regularity of sample paths of Volterra-L\'{e}vy processes.
These processes are defined as stochastic integrals
$$ M(t)=\int_{0}^{t}F(t,r)dX(r), \   \ t \in \mathds{R}_{+}, $$
where $X$ is a L\'{e}vy process and $F$ is a deterministic real-valued function. We derive the spectrum of singularities and a result on the 2-microlocal frontier of $\{M(t)\}_{t\in [0,1]}$, under regularity assumptions on the function $F$.

\section{Introduction and main results}
\subsection{Volterra Processes} \label{Semimartingales}
A Volterra process is given by
\be \label{volterra}
 M(t)=\int_{0}^{t}F(t,r)dX(r), \   \ t \in \mathds{R}_{+},
\ee
where $\{X(t)\}_{t\geq0}$ is a semimartingale and $F(t,r)$ is a bounded deterministic real-valued function of two variables which sometimes is called a kernel.
One of the questions addressed in the research of Volterra and related processes is studying their regularity
properties.
It is also the  main goal of this paper. Before we describe our results let us give a short introduction to this area.
First, let us note that one-dimensional fractional processes, which are the close relative of Volterra processes,
have been extensively studied in the literature. One-dimensional fractional processes are usually defined by
\begin{eqnarray}
\label{frac:1}
X(t)&=& \int_{-\infty}^{\infty} F(t,r) dL(r),
\end{eqnarray}
where $L(r)$ is some stochastic process and $F(t,r)$ is some specific kernel. For example in the case of $L(r)$
being a two-sided standard Brownian motion and $F(t,r)=\frac{1}{\Gamma(H+1/2)}\left((t-s)^{H-1/2}_+ -(-s)^{H-1/2}_+\right)$, $X$ is called fractional Brownian motion with Hurst index $H$ (see  e.g. Chapter 1.2 of~\cite{oks08} and Chapter 8.2 of~\cite{nualart}). It is also known that the fractional Brownian motion with Hurst index $H$ is H\"older
continuous with any exponent less than $H$ (see e.g.~\cite{mandel.vanness}). Another prominent example is the case of the fractional $\alpha$-stable L\'{e}vy process which can be also defined via~(\ref{frac:1}) with $L(r)$ being a two-sided $\alpha$-stable L\'{e}vy process and
\bd
F(t,r)=a\{(t-r)^d_{+}-(-r)^d_{+}\}+b\{(t-r)^d_{-}-(-r)^d_{-}\}.
\ed
Takashima in~\cite{A8} studied the regularity properties of the sample paths of this process.
Another well-studied process is the so-called
fractional L\'{e}vy process, which again is defined via~(\ref{frac:1}) with the following kernel function
 \be \label{Kernel-frac-Levy}
F(t,r)=\frac{1}{\Gamma(d+1)}[(t-r)^d_{+}-(-r)^d_{+}],
\ee
and $L(r)$ being a two-sided L\'evy process.
For example,  Marquardt in \cite{A4} studied the case where $E[L(1)]=0$, $E[L(1)^2 ]<\infty$ and $L$ does not have a Brownian component. It was proved in \cite{A4}, that the sample paths of the
fractional L\'{e}vy process are $P$-a.s. local H\"{o}lder continuous of any order $\beta < d$. Later in \cite{MytnikNeuman}, it was proved that the sample paths of the fractional L\'{e}vy process are $P$-a.s. local H\"{o}lder continuous of order $d$.
The regularity properties of the analogous multidimensional processes have been also studied.
For example, consider the process
\be \label{MRT}
\hat{M}(t)=\int_{\mathds{R}^m}F(t,r)L(dr), \   \ t \in \mathds{R}^N,
\ee
where $L(dr)$ is some random measure and $F$ is a real valued function of two variables. A number of important results have been derived recently by Ayache, Roueff and Xiao in \cite{aya1}, \cite{aya2},
on the regularity properties of $\hat{M}(t)$ for some particular choices of $F$ and $L$.
As for the earlier work on the subject we can refer to K\^{o}no and Maejima in \cite{kono} and \cite{A6}.
We should also mention the book of Samorodnitsky and Taqqu \cite{A3} and the work of Marcus and Rosi\'{n}sky in \cite{marcus} where the regularity properties of processes related to $\hat{M}(t)$ in (\ref{MRT}) were also studied.

\subsection{Functions of Smooth Variation as Kernel Function of Volterra Processes} \label{smoothvar}
In this section we make our assumptions on the kernel function $F(s,r)$ in (\ref{volterra}). We review the notation and definitions from \cite{MytnikNeuman} that are relevant to this context.
First we introduce the following notation. \\\\
We denote 
\bd
F^{(n,m)}(s,r) \equiv \frac{\partial^{n+m} F(s,r) }{\partial s^{n}  \partial r^{m}}, \ \ \forall n,m=0,1,\ldots.
\ed
We also define the following sets in $\mathds{R}^2$:
\bd
E=\{(s,r):-\infty<r \leq s < \infty \},
\ed
\bd
\tilde{E}=\{(s,r):-\infty<r < s < \infty \}.
\ed
We denote by $K$ a compact set in $E$, $\tilde{E}$ or $\mathds{R}$, depending on the context.
We define the following spaces of functions that are essential for the definition of functions
of \emph{smooth variation}.
\begin{definition} \label{CE}
Let $\mathbb{C}_{+}^{(k)}(E)$ denote the space of functions $F$ from the domain $E$ in $\mathds{R}^2$ to $\mathds{R^1}$ satisfying  \\
1. $F$ is continuous on $E$;\\
2. $F$ has continuous partial derivatives of order $k$ on $\tilde{E}$. \\
3. $F$ is strictly positive on $\tilde{E}$.
\end{definition}
Note that functions of smooth variation of one variable have been studied extensively in the literature; \cite{reg} is the standard reference for these and related functions.
Here we generalize the definition of functions of smooth variation to functions on $\mathds{R}^2$.
\begin{definition} \label{smt.var2.0}
Let $F \in \mathbb{C}_{+}^{(k)}(E)$ satisfying for every compact set $K \subset \mathds{R}$ \\
a)
\bd
\lim_{h\downarrow 0}\sup_{ t \in K}\bigg| \frac{h F^{(0,1)}(t,t-h)}{F(t,t-h)}+\rho\bigg|=0,
\ed
b)
\bd
\lim_{h\downarrow 0}\sup_{t \in K}\bigg| \frac{h F^{(1,0)}(t+h,t)}{F(t+h,t)}-\rho\bigg|=0,
\ed
c)
\bd
\lim_{h\downarrow 0}\sup_{t \in K}\bigg| \frac{h^{j} F^{(j-1,1)}(t,t-h)}{F(t,t-h)}+\rho(\rho-1)...(\rho-j+1)\bigg|=0, \ \ j=2,...,k,
\ed
d)
\bd
\lim_{h\downarrow 0}\sup_{t \in K}\bigg| \frac{h^2 F^{(0,2)}(t,t-h)}{F(t,t-h)}-\rho(\rho-1)\bigg|=0.
\ed
Then $F$ is called a function of smooth variation of index $(\rho,k)$ at the diagonal and is denoted as
$f \in SR_\rho^k(0+)$.
\end{definition}
The trivial example for a function of smooth variation $SR_\rho^k(0+)$, for all $k\in\mathds{N}$, is $f(t,r)=(t-r)^{\rho}$.
Another example would be $f(t,r)=(t-r)^{\rho}|\log(t-r)|^{\eta}$ where $\eta\in \mathds{R}$. \\\\
In \cite{MytnikNeuman} the following results for the sample path properties of Volterra processes were proved.
\begin{theorem} \label{thm1}
Let $F(t,r)$ be a function of smooth variation of index $(d,2)$ and let $\{X(t)\}_{t\geq 0}$ be a semimartingale.
Denote $\Delta_{X}(s)=X(s)-X(s-)$. Define
\begin{displaymath}
M(t) =  \int_{0}^{t}F(t,r)dX(r), \ \ t\geq 0.
\end{displaymath}
Then,
\begin{itemize}
  \item [{\bf (a)}]
\begin{displaymath}
\lim_{h\downarrow 0}  \frac{M(s+h)-M(s)}{F(s+h,s)}=\Delta_{X}(s) , \ \ \forall s \in [0,1], \ \ P-\rm{a.s.},
\end{displaymath}
 \item [{\bf (b)}]
\begin{displaymath}
\lim_{h\downarrow 0} \ \  \sup_{0< s<t <1,\ \ |t-s|\leq h}\frac{|M(t)-M(s)|}{F(t,s)} =  \sup_{s\in [0,1]} |\Delta _{X}(s)|, \   \ P-\rm{a.s.}
\end{displaymath}
\end{itemize}
\end{theorem}
Part (a) of Theorem \ref{thm1} gives us information about the regularity of the sample paths of $M$. It also shows that at the point of jump $s$, the increment of the process behaves like $F(s+h,s)\Delta_X(s)$. Part (b) of Theorem  \ref{thm1} gives uniform in time bound on the increments of the process $M$.
 \\\\
We will use the results of Theorem \ref{thm1} throughout this work.
\subsection{Multifractal Analysis and 2-Microlocal Analysis} \label{Microlocal} \label{Multi}
We now discuss briefly the area of \emph{multifractal analysis} and \emph{2-microlocal analysis} of stochastic processes. In this section we give some necessary
definitions and results. Later on we will use the material of this section for analyzing the multifractal nature of
a particular Volterra process.\\\\
We begin with the definition of \emph{pointwise regularity} which is used for analyzing the regularity of
not necessarily differentiable sample paths.
\begin{definition} \label{pointwise} \textbf{Pointwise Regularity} \mbox{\boldmath$C^{l}(t_0)$}
[See e.g. Section 1 of \cite{jaffard2}]. Let $t\in \mathds{R}$ and let
$l$ be a positive real number. A function $g(t)$ is $C^{l}(t_{0})$ if there exists a constant $C>0$ and a
polynomial $P_{t_{0}}$ of degree at most $\lfloor l \rfloor$ such that in the neighborhood of $t_{0}$,
\bd
|g(t)-P_{t_{0}}(t)|\leq C|t-t_0|^l.
\ed
\end{definition}
The following number is called the H\"{o}lder exponent of $g$ at $t_0$:
\be \label{Holderexponent}
h_g(t_{0})=\sup\{l:g\in C^{l}(t_0)\}.
\ee
Multifractal analysis deals with the study of the sets $S^g_h$, that contain the points where a function $g$ has a
given H\"{o}lder exponent $h$. Denote by $d(h)$ the Hausdorff dimension of $S^g_h$.
The function $h\rightarrow d(h)$ is called the \emph{spectrum of singularities} of $g$. Functions with a non-constant
$d(h)$ are called \emph{multifractal functions}.\\\\
Multifractal functions were first introduced in \cite{frisch} in the area of physics called fully developed turbulence.
Later on the spectrum of singularities of various functions was studied (see \cite{jaffard2}, \cite{wave}, \cite{mendel} and references therein). \\\\
The spectrum of singularities has been also investigated for some random processes. Jaffard in \cite{jaffard} studied the multifractal nature of general L\'{e}vy processes (for the Brownian motion case see e.g. \cite{Csrgo}, \cite{Deh}, \cite{Steven}). Barral and L{\'e}vy V{\'e}hel studied  in \cite{Barral-Vehel2004} the spectrum of singularities for a class of additive processes with correlated non-stationary increments. The spectrum of singularities of L\'{e}vy processes in multifractal time was studied in \cite{Barral-Seuret2007}. The multifractal structure of super-Brownian motion was studied by Perkins and Taylor in \cite{Perkins98}. Finaly the multifractal spectrum for a class of superprocess with stable branching in one dimension was derived in \cite{Mytnik-Wachtel2012}.
In particular, the case of L\'{e}vy processes has attracted considerable interest.
Before we describe the results in \cite{jaffard}, we define the following parameter
\be \label{beta}
\beta=\inf\bigg \{ \gamma \geq 0 : \int_{|x|\leq 1} |x|^\gamma \pi(dx) <\infty\bigg\}, \ \ \beta'=\left \{ \begin{array}{ll}
\beta  \ \ \textrm{    if   } Q=0, & \\
2 \  \  \textrm{   if   } Q\not = 0.
\end{array} \right.
\ee
where $\pi(dx)$ is a L\'{e}vy measure. Clearly $\beta \in [0,2]$. Here $Q$ denotes the diffusion coefficient of the L\'{e}vy processes. In the case where $Q=0$, the L\'{e}vy processes have no Brownian component.      \\\\
Let
\bd
d_\beta(h)=
\left \{ \begin{array}{ll}
\beta h \ \ \textrm{    if   } h \in [0,\beta'], & \\
1 \  \  \textrm{  if    }  h =1/\beta', & \\
-\infty \  \  \textrm{  else;    }
\end{array} \right.
\ed
\be \label{Cj}
C_j=\int_{2^{-j-1}\leq |x|\leq 2^{-j}}\pi(dx).
\ee
Jaffard in \cite{jaffard} showed the following:
\begin{theorem} \textbf{[Theorem 1 in \cite{jaffard}]} \label{Thm-Jaffard}
Let $X(t)$ be a L\'{e}vy process of L\'{e}vy measure $\pi(dx)$ satisfying $\beta>0$ and
\be \label{sumcond}
\sum 2^{-j}\sqrt{C_j \log(1+C_j)}<\infty.
\ee
\begin{itemize}
  \item The spectrum of singularities of almost every sample path of $X(t)$ is $d_\beta(h)$.
  \item If $\beta = 0$ but $\pi(\mathds{R})=\infty$, then for each $h$, with probability 1, $d(h)=0$.
\end{itemize}
\end{theorem}
\begin{remark}
Later in \cite{Balanca}, it was shown that assumption (\ref{sumcond}) in Theorem \ref{Thm-Jaffard} can be removed.
\end{remark}
Durand in \cite{durand} extended the result of Jaffard by describing the spectrum of singularities
of L\'{e}vy processes in a more general context.  \\\\
It was shown in \cite{Herbin-Levy2009} that multifractal analysis does not give a complete description of the local regularity in some cases.
It is also known that it lacks the stability needed under the transformation of pseudo-differential operators.
The way to overcome these problems is the so-called 2-microlocal analysis. In \cite{Bony}, \cite{Jaffard-Meyer1996} and \cite{Levy-Sueret2004} the 2-microlocal analysis was introduced in the study of deterministic functionals. A time domain characterization of the 2-microlocal spaces was introduced in \cite{Kolwankar-Levy2002} for deterministic functions.
Lately Herbin and L\'{e}vy-V\'{e}hel in \cite{Herbin-Levy2009} and Balan\c{c}a and Herbin in \cite{Balanca-Herbin2012}, developed a stochastic approach to $2$-microlocal analysis. In the rest of this section we give some necessary definitions and results from the field of 2-microlocal analysis. The definitions are taken from Sections 2 and 3 in \cite{Kolwankar-Levy2002} and from Section 2 in \cite{Balanca-Herbin2012}.
We begin definitions of some functional spaces, which are known as \emph{2-microlocal spaces}.
\paragraph{Notation.} Let $t_0\in \re$ and $h>0$. We denote by $B(t_0,h)$ a ball centered $t_0 $ with radius $h$. \\
Denote by $\dim(A)$ the Hausdorff dimension of a set $A\subseteq \mathds{R}$.
\begin{definition} \label{Def-micro-space}  \textbf{2-Microlocal Space} \mbox{\boldmath$C^{\sigma,s'}(t_0)$}
Let $t_0\in \re$, $s'\leq 0$ and $\sigma \in (0,1)$. A continuous function $g:\re\rr\re^d$ belongs to $C_{t_0}^{\sigma,s'}$ if there exist $C>0,\rho>0$ and a polynomial $P_{t_0}$ of degree at most $\lfloor \sigma -s'\rfloor$ such that for all $u,v\in B(t_0,h)$,
\bq \label{micro1}
|(g(u)-P(u))-(g(v)-P(v))|\leq C|u-v|^{\sigma}(|u-t_0|+|v-t_0|)^{-s'}.
\eq
\end{definition}
Note that in the case where $\sigma \in [0,1)$ and $\sigma-s'\in[0,1)$, then $P_{t_0}\equiv 0$ and (\ref{micro1}) becomes
\bn
|g(u)-g(v)|\leq C|u-v|^{\sigma}(|u-t_0|+|v-t_0|)^{-s'}, \ \forall u,v\in B(t_0,h).
\en
The \emph{2-microlocal frontier} of a function $g$ at $t_0$ is defined as the map $s'\mapsto\sigma_{g,t_0}^{(s')}$ such that
\bq\label{def-frontier}
\forall s'\leq 0, \ \sigma_{g,t_0}(s')=\sup\{\sigma\in [0,1): g\in C_{t_0}^{\sigma,s'}\}.
\eq
\paragraph{Notation.}
Let $C^{\alpha}(E)$ be the space of functions with global H\"{o}lder $\alpha$ on $E\subset \re$.    \medskip \\
Assume $g\in C^\eps(\re)$ for some $\eps>0$. Then by Proposition 3 in \cite{Levy-Sueret2004}, the H\"{o}lder exponent of $g$ at $t_0$ satisfies
\bq \label{holder-frontier}
h_g(t_0)=-\inf\{s',\sigma_{g,t_0}(s')\geq 0\},
\eq
with the convention that $h_g(t_0)=\infty$ if $\sigma_{g,t_0}(s')>0$ for all $s'$.  \medskip \\
In \cite{Balanca}, Balan\c{c}a studied the following sets
\bn \label{def-E-sigma-sets}
E_{h}=\{t_0\in S_h:\forall s'\leq 0, \sigma_{X,t}(s')=(h+s')\wedge 0 \}.
\en
Let $\mathcal{O}$ be the collection of all nonempty open sets of $\re$. It was shown in \cite{Balanca} that the sample paths of a L\'{e}vy process $X$ almost surely satisfy
\bq \label{res-balanca}
\forall V\in \mathcal{O},  \ \dim(\tilde E_h\cap V ) = d_{\beta}(h).
\eq
Note that (\ref{holder-frontier}) with (\ref{res-balanca}) generalizes the result in Theorem \ref{Thm-Jaffard} in the sense that assumption (\ref{sumcond}) can be removed.
Finally in \cite{Balanca}, the following linear fractional stable processes was studied
\be \label{lin-frac-stable}
M_t =  \int_{\re}\{(t-r)^{H-1/\alpha}_{+}-(-r)^{H-1/\alpha}_{+}\}M_{\alpha,\beta}(dr),
\ee
where $M_{\alpha,\beta}$ is an $\alpha$-stable measure with $\alpha\in [1,2)$ and $H\in(0,1)$.
Denote
\bq \label{def-E-sigma-sets2}
E_{\sigma,s'}=\{t_0\in\re_{+}:\forall u'>s', M_\cdot \in C_{t_0}^{\sigma,u'} \textrm{ and } \forall u'<s', M_\cdot \not \in C_{t_0}^{\sigma,u'}\}.
\eq
It was shown in \cite{Balanca} that the process $M$ satisfies almost surely for all $\sigma\in[H-1/\alpha-1,H-1/\alpha]$,
\bd
\forall V\in \mathcal{O},  \ (E_{\sigma,s'}\cap V ) =
\left \{ \begin{array}{ll}
\alpha(s-H)+1  \ \ \textrm{    if   } s \in [H-1/\alpha,H], & \\
-\infty, \  \  \textrm{  otherwise,    }
\end{array} \right.
\ed
where $s=\sigma - s'$.
\begin{remark} The proof of the result in \cite{Balanca} relies heavily on the following property (see Theorem 1.1 in \cite{Jaffard-Meyer1996}),
\bq \label{PSEODO}
\forall a>0, \ f\in C_t^{\sigma,s'} \Longleftrightarrow I^{a}f\in C_t^{\sigma+a,s'},
\eq
where $I^{a}f$ denotes the fractional integral of $f$ of order $a$.
\end{remark}
\subsection{Main Results} \label{TheoremsSection}
From now on we consider a semimartingale $\{X(t)\}_{t\geq 0}$ such that $X(0)=0$ $P$-a.s. Without
loss of generality we assume further that $X(0-)=0$, $P$-a.s. \\
Next we provide a detailed description of the pointwise regularity of Volterra processes.
Before we present our results, we will need the following definitions.
In what follows define
\be \label{slx}
S_l^X=\{t\in[0,1]:  \ h_X(t)=l\}.
\ee
We define the following space of functions which is a generalization of the 2-microlocal space introduced
in Definition \ref{Def-micro-space}.
\begin{definition} \label{Def-F-micro-space}  \textbf{\mbox{\boldmath$C_F^{d,s'}(t_0)$: 2-Microlocal Space with Gauge Function $F$} }
Let $t_0\in \re$, $d\in(0,1)$, $s'\leq 0$, and $F\in SR_d^{2}(0+)$. A continuous function $g:\re\rr\re^d$ belongs to $C_{F}^{d,s'}(t_0)$ if there exist $C>0,h>0$ and a polynomial $P_{t_0}$ at degree at most $\lfloor d -s'\rfloor$ such that,
\bq \label{micro}
|(g(u)-P(u))-(g(v)-P(v))|\leq C F(u,v)(|u-t_0|+|v-t_0|)^{-s'}, \  \ \forall  u,v\in B(t_0,h).
\eq
\end{definition}
Note that in the case where $d-s'\in[0,1)$, we have $P_{t_0}\equiv 0$ and (\ref{micro}) becomes
\bn
|g(u)-g(v)|\leq C F(u,v)(|u-t_0|+|v-t_0|)^{-s'}, \ \forall u,v\in B(t_0,h).
\en
We define the \emph{H\"{o}lder exponent with gauge function $F$} of $g$ at $t_0$ to be
\be \label{Def-F-Holderexponent}
\bar{h}^d_g(t_{0})=d-\inf\{s':g\in C_F^{d,s'}(t_0)\}.
\ee
Our goal is to find the Hausdorff dimension of the sets
\be \label{Sets-QM}
Q^{M}_{l,d}=\{t\in[0,1]:\bar{h}^d_ M(t_{0})=l+d\},
\ee
where $M$ is some Volterra process to be defined later. Derivation of the Hausdorff dimension of the sets $Q^{M}_{l,d}$ will be a generalization of the result in \cite{Balanca} where the sets in (\ref{def-E-sigma-sets2}) were studied.
The following theorem provides a lower bound on the H\"{o}lder exponent with gauge function $F$ of a Volterra process $M$ at a time $t$ when the semimartingale $X(t)$ has H\"{o}lder exponent $l$.
\begin{theorem} \label{Corollary- point1}
Let $d\in(0,1)$ and let $\{M(s)\}_{s\geq 0}$ be a Volterra process
\bd
M(s)=\int_0^sF(s,r)dX(r) , \  s \geq 0,
\ed
where $\{X(r)\}_{r\geq 0}$ is a semimartingale.  Let $l\geq 0$ and $F\in SR^{\lfloor d+l \rfloor +2}_d(0+)$. Then
\bd
\bar{h}^d_M(t) \geq l+d, \ \ \forall t\in S_l^X, \ \  P-\rm{a.s.}
\ed
\end{theorem}
Whenever $X$ is a L\'{e}vy process one can deduce a result on the spectrum of
singularities for $M$ which in this case is called the Volterra-L\'{e}vy process.
\begin{theorem} \label{Theorem-Spect}
Let $\{X(t)\}_{t\geq 0}$ be a L\'{e}vy process without a Brownian component, of L\'{e}vy measure $\pi(dx)$ satisfying $\beta>0$.
Fix $d\in(0,1)$ and let $F\in SR^{\lfloor d+1/\beta \rfloor +2}_d(0+)$. Let $\{M(t)\}_{t\geq0}$ be a Volterra-L\'{e}vy process
\bd
M(t)=\int_0^tF(t,r)dX(r)\  \ \ ,t\geq 0.
\ed
Then
\bd
\dim (Q^{M}_{v,d}) = \beta v, \ \  \forall v  \in [0,1/\beta), \ \ P-\rm{a.s.}
\ed
\end{theorem}
The definition of $\beta$ is given in (\ref{beta}).  \\\\
The following theorem provides us with the 2-microlocal frontier of the Volterra-L\'{e}vy process $M$ at a time $t$ when the L\'{e}vy process $X(t)$ has H\"{o}lder exponent $l$.
\begin{theorem} \label{point1}
Let $\{X(t)\}_{t\geq 0}$ be a L\'{e}vy process without a Brownian component, of L\'{e}vy measure $\pi(dx)$ satisfying $\beta>0$.
Fix $d\in(0,1)$ and let $F\in SR^{\lfloor d+1/\beta \rfloor +2}_d(0+)$. Let $\{M(t)\}_{t\geq0}$ be a Volterra-L\'{e}vy process
\bd
M(t)=\int_0^tF(t,r)dX(r)\  \ \ ,t\geq 0.
\ed
Then
\bd
\sigma_{M,t}(-l) = d, \ \ \forall t\in S_l^X, \ \  P-\rm{a.s.}
\ed
\end{theorem}
\begin{remark} \label{remark}
Note that the proofs of spectrum of singularities of integral operators like in Theorem 2 in \cite{Balanca} relies heavily on the fact that pseudo-differential operators are invertible. The proof of Theorem \ref{Theorem-Spect} does not rely on an inverse transform. In fact, in the case where
\be \label{examp}
 F(t,r)=(t-r)^{d}|\log(t-r)|^\eta, \  d\in(0,1), \  \eta>0,
\ee
the integral transform is not even one-to-one. We refer to Chapter 6, Section 34.2, subsection 32.6 in \cite{samko} for further discussion on this example.
\end{remark}
Let us survey a few applications and examples to the results in this section. Assume that $X$ is an $\alpha$-stable process with $\alpha\in (1,2)$ and
\bn
M(t)=\int_{\re}[(t-r)^{H-1/\alpha}_{+}-(-r)^{H-1/\alpha}_{+}]dX(r),  \ t\in \re_{+},
\en
where $H\in (1/\alpha,1)$. Note that $M$ is a linear fractional stable processes. From integration by parts (see for example Proposition 2 in \cite{Balanca}) and the decomposition in the proof of Theorem 1.6 in \cite{MytnikNeuman}, we immediately get
\bq \label{dec-frac-stable}
M(t)&=&C(H)\int_{\re}X(u)[(t-u)_{+}^{H-1/\alpha-1}-(-u)^{H-1/\alpha-1}_{+}]du   \\
&=& C(H)\int_{0}^{t}X(u)(t-u)^{H-1/\alpha-1}+ C(H)\int_{-\infty}^{0}X(u)[(t-u)^{H-1/\alpha-1}-(-u)^{H-1/\alpha-1}]du, \nonumber \\ 
&& \ \forall t\in(0,1),  \ P-\rm{a.s.} \nonumber
\eq
Now argue in the proof of Theorem 1.6 in \cite{MytnikNeuman}, that the regularity of $M$ is determined by the first term on the right hand side of (\ref{dec-frac-stable}). Recall the definition of $E_{\sigma,s'}$ in (\ref{def-E-sigma-sets2}). In this example we get that $Q^{M}_{v,H-1/\alpha}= E_{H-1/\alpha,-v}\cap [0,1], \ \forall v>0$, and the results of Theorem \ref{Theorem-Spect} coincide with the results of Theorem 2 in \cite{Balanca}. \medskip \\
Assume that $X$ is a L\'{e}vy process such that $E[X(1)]=0$, $E[X(1)^2 ]<\infty$, and that $X$ does not have a Brownian component. Define $F$ to be as in (\ref{Kernel-frac-Levy}). By the same argument that was used for the linear fractional stable processes, it is straightforward to show that the results of Theorem \ref{Theorem-Spect} coincide with Proposition 3 in \cite{Balanca}. \medskip \\
Finally, let $M$ be a Volterra process where $X$ is an $\alpha$-stale L\'{e}vy process where $\alpha\in(0,2)$ and $F$ as in (\ref{examp}). This type of process is a special case of the moving average $\alpha$-stable processes that were studied in \cite{A3}. Note that in this example $M\in C_F^{d,s'}(t)$, means that there exists $C(\omega)>0,h(\omega)>0$ and a polynomial $P_{t_0}(\omega)$ at degree at most $\lfloor d -s'\rfloor$ such that,
\bn
|(M(u)-P(u))-(M(v)-P(v))|\leq C (u-v)^{d}|\log(u-v)|^\eta(|u-t_0|+|v-t_0|)^{-s'}, \  \ \forall  u,v\in B(t_0,h).
\en
From Theorem \ref{Theorem-Spect} follows
\bn
\dim\{t\in[0,1]:\bar{h}^d_ M(t_{0})=v+d\}=\alpha v, \ \  \forall v  \in [0,1/\alpha), \ \ P-\rm{a.s.}
\en
Recall that by remark \ref{remark}, we cannot use the fact that the integral transform in (\ref{volterra}) is invertible, and we need to use the methods which were developed in this work.
Note that similar results can be derived for the moving average L\'{e}vy processes that was studied in \cite{A4}, for the special case where $F$ is defined by (\ref{examp}). We also notice from the theorems and the examples in this section that the regularity of the process $M$ is determined by the regularity of $X$ and the regularity of the function $F$ on the diagonal.  \\\\
In Section \ref{sec-prop-spaces} we study the properties of 2-microlocal spaces and the H\"{o}lder exponent with gauge functions. In Section \ref{pointChap} we prove Theorem \ref{Corollary- point1}. Section \ref{Aplications} is devoted to the derivation of a lower bound of the Hausdorff dimension of the sets $Q_{v,d}^M$ and then to the proof of Theorem \ref{point1}. In Section \ref{Aplications2} we derive the upper bound of the Hausdorff dimension of the sets $Q_{v,d}^M$ and prove Theorem \ref{Theorem-Spect}.
\section{Properties of 2-Microlocal Spaces and H\"{o}lder Exponent with Gauge Functions} \label{sec-prop-spaces}
In this section we study the properties of 2-microlocal spaces and H\"{o}lder exponents with gauge functions of smooth variation. The following theorem gives some basic properties of the 2-microlocal spaces with gauge function $F$, and their relations to the regularity spaces which were introduced in Section \ref{Microlocal}.
\begin{theorem}  \label{thm-2-micro}
Let $F\in SR_d^{2}(0+)$, $d\in(0,1)$ and $s\leq 0$. Then for every $t_0\in\re$ we have
\begin{itemize}
  \item [\bf{(a)}]If $\tilde s <s$, then $C_{F}^{ d,\tilde s}(t_0) \subset  C_{F}^{ d, s}(t_0)$,
  \item [\bf{(b)}]  $C_{F}^{ d, s}(t_0) \subset C^{d-\eps,s}(t_0), \ \forall   \eps\in (0,d)$,
  \item [\bf{(c)}] $C^{d+\eps,s}(t_0) \subset C_{F}^{ d, s}(t_0), \ \forall \eps\in (0,\infty)$,
  \item [\bf{(d)}] $ C_{F}^{ d, s}(t_0)\subset C^{d-s-\eps}(t_0) , \ \forall \eps\in (0,d)$.
\end{itemize}
\end{theorem}
Before we prove Theorem \ref{thm-2-micro}, we introduce the following auxiliary lemma.
\begin{lemma}\label{lemma-aux1}
Let $F(t,r)\in SR_d^{\lfloor l+d\rfloor+2}(0+)$. Then for every $\eps>0$ and $-\infty <a<b<\infty$ we have
\begin{itemize}
\item [{\bf (a)}]
\begin{displaymath}
\lim_{h\downarrow 0} \ \  \sup_{a< v<u <b,\ \ |u-v|\leq h}(u-v)^{k+j-d+\eps} F^{(k,j)}(u,v)=0, \ \forall j=0,1,2,\ k\in 0,...,\lfloor l+d\rfloor+2-j,
\end{displaymath}
\item [{\bf (b)}]\begin{displaymath}
\lim_{h\downarrow 0} \ \  \sup_{a< v<u <b,\ \ |u-v|\leq h}\frac{(u-v)^{d+\eps}}{F(u,v)} =0.
\end{displaymath}
\end{itemize}
\end{lemma}
The proof of Lemma \ref{lemma-aux1} follows directly from the properties of smoothly varying functions.
\paragraph{Proof of Theorem \ref{thm-2-micro}} (a) follows from Definition \ref{Def-F-micro-space}. (b) follows immediately from Definitions \ref{Def-micro-space}, \ref{Def-F-micro-space} and Lemma \ref{lemma-aux1}(a). From Definitions \ref{Def-micro-space}, \ref{Def-F-micro-space} and Lemma \ref{lemma-aux1}(b) we get (c). (d) follows from Definition \ref{pointwise}  and \ref{Def-micro-space} and Lemma \ref{lemma-aux1}(a).  \qed \medskip \\
The following proposition shows the connection between the H\"{o}lder exponent with gauge function $F$, the H\"{o}lder exponent defined in (\ref{Holderexponent}) and the 2-microlocal frontier.
\begin{proposition} \label{proposition-prop-exp}
Let $d\in(0,1)$, and $l\geq 0$.  For every function $g$ we have
\begin{itemize}
  \item [\bf{(a)}] $\{t:\bar h^d_g(t) =l\}\subset  \{t:h_g(t)\geq l+d\}$,
  \item [\bf{(b)}] $\dim\{t:\bar h^d_g(t) =l\}\leq \dim \{t:h_g(t)\geq l+d\}$,
  \item [\bf{(c)}] $\{t:\bar h^d_g(t) =l\} \subset \{t:\sigma_{g,t}(-l)=d\}$.
\end{itemize}
\end{proposition}
\begin{proof}
From (\ref{Holderexponent}), (\ref{Def-F-Holderexponent}), and Theorem \ref{thm-2-micro}(d) we immediately get (a). (b) follows directly from (a). The proof of (c) follows the same lines as the proof of Theorem \ref{point1}, hence it is omitted.
\end{proof}
\\\\
Recall the definition of the local H\"{o}lder exponent (see for example in Section 2.2 of \cite{Kolwankar-Levy2002}). Let $\rho\in(0,1)$ and $A\subset \re$. We say that $g\in C_{\textrm{loc}}^\rho(A)$ if there exists a constant $C>0$ such that
\bn
\frac{|g(x)-g(y)|}{|x-y|^\rho}\leq C, \ \forall x,y\in A.
\en
Define
\bd
\alpha_{\textrm{loc}}(g,t_0,h)=\sup \{\rho:g\in C_{loc}^\rho(B(t_0,h))\}.
\ed
Let $g$ be a continuous function. The local H\"{o}lder exponent of $g$ at $t_0$ is defined as
\bn
\alpha_{\textrm{loc}}(g,t_0)=\lim_{h\dr 0} \alpha_{\textrm{loc}}(g,t_0,h).
\en
The following corollary provides us with a lower bound on the local H\"{o}lder exponent of Volterra processes. This corollary follows directly from Theorem \ref{thm1} and Lemma \ref{lemma-aux1}.
\begin{corollary} \label{corollarty-local}
Let $d\in(0,1)$ and let $\{M(s)\}_{s\geq 0}$ be a Volterra process
\bd
M(s)=\int_0^sF(s,r)dX(r) , \  s \geq 0,
\ed
where $\{X(r)\}_{r\geq 0}$ is a semimartingale and $F\in SR^{2}_d(0+)$. Then
\be \label{cor-ineq}
\alpha_{\textrm{loc}}(M,t) \geq d, \ \ \forall t\in [0,1], \ \  P-\rm{a.s.}
\ee
Let $S$ be the set of jump times of $X$ on $t\in[0,1]$. Then (\ref{cor-ineq}) holds with equality for every $t\in S$.

\end{corollary}
While the lower bound on $\alpha_{\textrm{loc}}(M,\cdot)$ is immediate, the upper bound for $t\in [0,1]\cap S^c$ is a very interesting open problem.

\section{Proof of Theorem \ref{Corollary- point1}} \label{pointChap}
The goal of this section is to prove Theorem \ref{Corollary- point1}. First we recall some auxiliary results on the Volterra process
from \cite{MytnikNeuman}.  \\\\
We recall the integration by parts formula for Volterra processes (Lemma 2.1 in \cite{MytnikNeuman}). In the following lemma we refer to functions in $\mathds{C}^{(1)}(E)$, which is the space of functions from Definition  \ref{CE}, without the condition that $f>0$ on $\tilde{E}$. It is easy to show that functions of smooth variation satisfy the assumptions of this lemma.
\begin{lemma} \label{IntParts}
Let $X$ be a semimartingale such that $X(0)=0$ a.s. Let $F(t,r)$ be a function in $\mathds{C}^{(1)}(E)$ satisfying $F(t,t)=0$ for all $t\in \mathds{R}$. Denote $f(t,r)\equiv F^{(0,1)}(t,r)$.
Then,
\bd
\int_{0}^{t} F(t,r)dX(r)=-\int_{0}^{t} f(t,r) X(r) dr, \ \ P-\rm{ a.s.}
\ed
\end{lemma}
\textbf{Convention and Notation}\\
In what follows we use the notation $F(t,r)$ for a smoothly varying function of index $(d,\lfloor d+l \rfloor +2)$ (that is, $F \in SR^{\lfloor d+l \rfloor +2}_d(0+)$), where $d$ is some number in $(0,1)$ and $l\geq 0$. We denote by $f(t,r) \equiv F^{(0,1)}(t,r)$, a smooth derivative of index $(d-1,\lfloor d+l \rfloor +2)$, that is, $f  \in SD^{\lfloor d+l \rfloor +2}_{d-1}(0+)$. \\\\
Let $f(t,r) \in SD^{\lfloor d+l \rfloor +2}_{d-1}(0+)$ where $d\in(0,1)$.
Define the following function,
\bq \label{g-func}
f_{\dl}(t,v)&=&\frac{f(t+\dl,t+\dl-\dl v)}{f(t+\dl,t)} , \ \ t \in[0,1], \ \ v \geq 0,  \ \ \dl>0.
\eq
Let us state a lemma which deals with the properties of function $f_{\dl}$.
\begin{lemma} \label{seq56tag2}
Let $f(t,r) \in SD^{\lfloor d+l \rfloor +2}_{d-1}(0+) $ where $d\in(0,1)$ and $l\geq 0$. Let $f_{\dl}(t,v)$ be defined as in  (\ref{g-func}). Then,
\begin{displaymath}
\lim_{\dl\downarrow 0}\sup_{0 \leq t \leq 1} \bigg| \int_{0}^{1}|f_{\dl}(t,v) |dv - \frac{1}{d} \bigg|=0.
\end{displaymath}
\end{lemma}
For the proof of Lemma \ref{seq56tag2} we refer to Section 2 of \cite{MytnikNeuman}.
\paragraph{\textbf{Notation:}}
For $I\subset \re$, $D_\re(I)$ denotes the set of real valued c\`{a}dl\`{a}g functions on $I$.  Let $\Gamma\equiv \{\omega\in \Omega:X(\cdot,\omega)\in D_\re(\re_+)\}$. By the assumptions of Theorem \ref{Corollary- point1}, $P(\Gamma)=1$. \\\\
The proof of Theorem \ref{Corollary- point1} follows immediately from the following proposition and  (\ref{Def-F-Holderexponent}).
\begin{proposition} \label{propoint1}
Let $M(s)$ be as in Theorem \ref{point1}.
Then for any $\eps \in (0,l)$, there exists $C_{(\ref{propoint11})}=C_{(\ref{propoint11})}(\omega, t)$, $h_{\ref{propoint11}}=h_{\ref{propoint11}}(\omega, t)$ and a polynomial $P_t$ of degree at most $\lfloor l+d \rfloor$ such that
\be \label{propoint11}
|M(u)-P_t(u)-(M(v)-P_t(v))| \leq C_{\ref{propoint11}}F(u,v)(|u-t|^{l-\eps}+|v-t|^{l-\eps}), \ \  \forall u,v\in B(t,h_{\ref{propoint11}}), \ u>v>0 , \ \ t\in S_l^X, \ \ P-\rm{a.s.}
\ee
\end{proposition}
The following lemma helps us to prove Proposition \ref{propoint1}. 
\begin{lemma} \label{bound}
Let $g$ be a a function in $D_R[0,\infty)$ and suppose $h_g(t_{0})=l$, for some $l>0$.
For every constant $R>0$ and $\eps \in (0,l)$ there exist $C_{t_0}(R,\varepsilon)$ and a polynomial $P_{t_0}$, of degree at most $\lfloor l \rfloor$, such that
\be \label{globlem}
|g(t_0+\dl)-P_{t_0}(t_0+\dl)|\leq C_{t_0}(R,\eps)|\dl|^{l -\varepsilon} , \    \  \forall \dl \in [-R,R].
\ee
\end{lemma}
\begin{proof}
Since $h_g(t_{0})=l$, then for every $\eps \in (0,l)$ there exists
$\dl'>0$ such that
\be \label{ExpC}
|g(t_0+\dl)-P_{t_0}(t_0+\dl)|\leq C_{\ref{ExpC}}(t_0,\eps)|\dl|^{l-\varepsilon} ,  \    \  \forall \dl\in[-\dl',\dl'].
\ee
Now let $R>0$. If $0<R \leq \dl^{'}$ we are done. Suppose $R>\dl^{'}$. Then if we pick $\dl$ such that $|\dl|\leq \dl^{'}$, then the result follows from (\ref{ExpC}). Let $\dl^{'}<|\dl|<R$.
$g(x)$ is a locally bounded function on $[0,\infty)$; hence there exists a constant $K=K(R)>0$ such that
\begin{eqnarray*}
|g(t_0+\dl)-g(t_0+\dl^{'})| &\leq&\frac{K}{| \dl^{'} |^{l-\varepsilon}}|\dl^{'}|^{l-\varepsilon}.
\end{eqnarray*}
By defining the constant
\bd
C_{t_0,\eps}(R) \equiv \frac{K}{|\dl^{'}|^{l-\varepsilon}}+C_{\ref{ExpC}}(t_0,\eps),
\ed
we get (\ref{globlem}).
\end{proof}
\\\\
Assume for the rest of this section again that $\Gamma \subset \Omega$ is such that $P(\Gamma)=1$ and $X(\cdot, \omega)$ is c\`{a}dl\`{a}g for all $\omega \subset \Gamma$.
In what follows in this section, we will be working with $X(\cdot, \omega)$ for an arbitrary $\omega \subset \Gamma$. We omit $P$-a.s. notation as we will be working with a particular realization of $X$.
For the rest of the section, if it is not stated otherwise, we assume that $t$ is as in the statement of Theorem \ref{Corollary- point1}, that is,
\bd
t\in S^X_l.
\ed
All constants that appear in the rest of the section may depend on $t$ and $\omega$. \\\\
In the next subsection we will prove Proposition \ref{propoint1}.
\subsection{Proof of Proposition \ref{propoint1}}
The decomposition proved in the next lemma is crucial for the proof of Proposition \ref{propoint1}.
This decomposition is derived by a simple change in variables and hence its proof is omitted.
\begin{lemma} \label{pointdec}
Let
\bd
Y(t)=\int_{0}^{t}f(t,r)X(r)dr.
\ed
Then for every polynomial $P_t(\cdot)=P_t(\cdot,\omega)$ and any $u,v\in (0,1]$ with $u-v=\dl>0$, the following decomposition holds:
\bd
Y(u)-Y(v)= I_{1,1}(v,\dl)+I_{1,2}(v,\dl)-I_{2,1}(v,\dl)-I_{2,2}(v,\dl), \ \ \forall t\geq 0.
\ed
where
\begin{eqnarray}
I_{1,1}(v,\dl) &=& \dl\int_{0}^{1} f(v+\dl,v+\dl(1-z))[X(v+\dl(1-z))-P_t(v+\dl(1-z))]dz, \nonumber \\
I_{1,2}(v,\dl) &=& \dl\int_{0}^{1} f(v+\dl,v+\dl(1-z))P_t(v+\dl(1-z))dz, \nonumber \\
I_{2,1}(v,\dl) &=& \dl\int_{0}^{v/\dl}(f(v,v-\dl z)-f(v+\dl,v-\dl z))[X(v-\dl z)-P_t(v-\dl z)]dz, \nonumber \\
I_{2,2}(v,\dl)&=&\dl\int_{0}^{v/\dl}(f(v,v-\dl z)-f(v+\dl,v-\dl z))P_t(v-\dl z)dz. \nonumber
\end{eqnarray}
\paragraph{Notation.} In what follows we fix $t\in S_l^x$ and we restrict the polynomial $P_{t}$ from Lemma \ref{pointdec} to be of degree at most $\lfloor l \rfloor$.
\end{lemma}
In the next lemma we get a bound on the term $I_{1,2}(v,\dl)-I_{2,2}(v,\dl)$ in the decomposition of the increment $Y(u)-Y(v)$ from the previous lemma.
\begin{lemma} \label{BR12}
Let $I_{1,2}$ and $I_{2,2}$ be defined as in Lemma \ref{pointdec}. Then for every $\eps\in(0,l)$, there exist $C_{\ref{C1}}=C_{\ref{C1}}(\omega,t)>0$, $h_{\ref{C1}}=h_{\ref{C1}}(t)>0$ and a polynomial $P^1_t$ of degree at most $\lfloor l + d\rfloor$ such that,
\bq \label{C1}
|I_{1,2}(v,\dl)-P^1_t(v)-(I_{2,2}(v,\dl)-P^1_t(u))| &\leq & C_{\ref{C1}}F(u,v)(|v-t|^{l-\eps}+|u-t|^{l-\eps}), \nonumber \\
&& \ \forall u,v \in B(t,h_{\ref{C1}}), \  u-v=\dl>0, \ t > 0.
\eq
\end{lemma}
\begin{proof}
Recall that the function $F$ is the kernel function for the process $M$. Using the notation of Section \ref{smoothvar}, recall that  $F^{(0,k)}(t,r)=\frac{\partial^n}{\partial r^n} F(t,r)$ on the set $\tilde E$. We will also use the notation $\tilde F^{(k)}(t,r)$ for a function which satisfies
\bd
\frac{\partial^k}{\partial r^k} \tilde F^{(k)}(t,r)=F(t,r), \ \  \forall (t,r)\in  E.
\ed
From Definition \ref{smoothvar}, it follows that $F(v,v)=0$ for all $v\geq 0$. By integration we immediately get
\bq \label{new0}
I_{1,2}(v,\dl)-I_{2,2}(v,\dl)=\sum_{k=0}^{\lfloor l \rfloor}C_k\cdot t^{k}(\tilde{F}^{(k)}(v,0)-\tilde{F}^{(k)}(v+\dl,0)),
\eq
where $C_k=C_k(\omega)$ are some random constants independent of $t,v,\dl$.
Assume now that $t>0$ is fixed recall that $v>0$. Note that the first term in the summation in (\ref{new0}) determines the regularity of $I_{1,2}-I_{2,2}$. Since $F\in \mathbb{C}_{+}^{(\lfloor l+d \rfloor+2)}(E)$, we get by Taylor's Theorem that there exists a polynomial $P^1_{t}$ of order $ \lfloor l+d \rfloor$ and $h_{\ref{C1}}>0$ such that
\bq \label{neww2}
&&|F(u,0)-P^1_t(u)-F(v,0)+P^1_t(v)| \nonumber \\
&&\leq \frac{1}{\lfloor l \rfloor !}\bigg[\int_{t}^{u}F^{(0,\lfloor l+d \rfloor+1)}(z,0)(u-z)^{\lfloor l+d \rfloor}dz-\int_{t}^{v}F^{(0,\lfloor l+d\rfloor+1)}(z,0)(v-z)^{\lfloor l+d \rfloor}dz \bigg] \nonumber \\
&&\leq \frac{1}{\lfloor l \rfloor !}\sup_{z\in (t,1]}|F^{(0,\lfloor l+d \rfloor+1)}(z,0)|\bigg[\int_{t}^{v}[(u-z)^{\lfloor l+d \rfloor}-(v-z)^{\lfloor l+d \rfloor}]dz+\int_{v}^{u}(u-z)^{\lfloor l+d \rfloor}dz \bigg] \nonumber \\
&&\leq C(t)(|u-t|^{\lfloor l+d \rfloor+1}+|v-t|^{\lfloor l+d \rfloor+1}+|u-v|^{\lfloor l+d \rfloor+1}) \nonumber \\
&&\leq C(t)(u-v)(|u-t|^{\lfloor l+d \rfloor}+|v-t|^{\lfloor l+d \rfloor}) \nonumber \\
&&\leq C(t)F(u,v)(|u-t|^{l-2\eps}+|v-t|^{l-2\eps}),\ \forall v,u \in B(t,h_{\ref{C1}}), \ \eps\in (0,l),
\eq
where the last inequality follows from Lemma \ref{lemma-aux1}(b). From (\ref{new0}) and (\ref{neww2}), (\ref{C1}) follows.
\end{proof}
\\\\
In the next lemma we get the bound on the term of the $I_{1,1}(v,\dl)$ in the decomposition of the increment $Y(u)-Y(v)$.
\begin{lemma} \label{BR3}
Let $I_{1,1}$ be defined as in Lemma \ref{pointdec}. Then for any $\eps \in (0,l)$ there exists $C_{\ref{CR1}}=C_{\ref{CR1}}(\omega,t)>0$ and $h_{\ref{CR1}}=h_{\ref{CR1}}(t,\omega)>0$ such
that,
\be \label{CR1}
|I_{1,1}(v,\dl)| \leq C_{\ref{CR1}}F(u,v)(|u-t|^{l-\eps}+|u-t|^{l-\eps}), \  \forall u,v \in B(t,h_{\ref{CR1}}),  \ u-v=\dl>0, \ \ t\in S_l^X.
\ee
\end{lemma}
\begin{proof}
Let $f_\dl$ be defined as in (\ref{g-func}).
To bound $I_{1,1}$ we use the pointwise regularity
of $X$ at the points $t\in S_l^X$. By the definition of $S_l^X$, Lemma \ref{bound} and simple algebra we get that for every $\varepsilon \in (0,l)$ there exists $C_{\ref{CR1}}= C_{\ref{CR1}}(\omega, t)$ and $h_{\ref{CR1}}= h_{\ref{CR1}}(\omega, t)$ such that
\begin{eqnarray} \label{I1}
\frac{|I_{1,1}(v,\dl)|}{\dl |f(v+\dl,v)|}&\leq& \int_{0}^{1} |f_{\dl}(v,z)||X(v+\dl(1-z))-P_t(v+\dl(1-z))|dz  \nonumber \\
&\leq&C(t)\int_{0}^{1} |f_{\dl}(v,z)||v+\dl(1-z)-t|^{l-\eps}dz  \nonumber \\
&\leq&C(t)\bigg[|v-t|^{l-\eps}\int_{0}^{1} |f_{\dl}(v,z)|dz+ \dl^{l-\eps}\int_{0}^{1} |f_{\dl}(v,z)|dz \bigg] \nonumber \\
&\leq& C(\omega,t)(|u-t|^{l-\eps}+|v-t|^{l-\eps})\int_{0}^{1} |f_{\dl}(v,z)|dz \nonumber \\
&\leq & C_{\ref{CR1}}(|u-t|^{l-\eps}+|v-t|^{l-\eps}),\ \  \forall u,v\in B(t,h_{\ref{CR1}}), \ u>v \ , \ \ t\in S_l^X,
\end{eqnarray}
where the last inequality follows from Lemma \ref{seq56tag2}.
From (\ref{I1}) and  Definition \ref{smt.var2.0}(a), (\ref{CR1}) follows.
\end{proof} \\\\
In the next lemma we get the bound on the term $I_{2,1}(v,\dl)$ in the decomposition of the increment $Y(u)-Y(v)$.
\begin{lemma} \label{BR4}
Let $I_{2,1}$ be defined as in Lemma \ref{pointdec}.
Then for any $\eps \in (0,l)$ there exists a polynomial $P^2_t$ of degree at most $\lfloor l+d \rfloor$, and constants $C_{\ref{CR2}}=C_{\ref{CR2}}(\omega,t)>0$, $h_{\ref{CR2}}=h_{\ref{CR2}}(\omega,t)>0$ such that,
\bq \label{CR2}
|I_{2,1}(u,\dl)-P^2_t(u)-(I_{2,1}(v,\dl)-P^2_t(v))| &\leq& C_{\ref{CR2}}F(u,v)(|u-t|^{l-\eps}+|v-t|^{l-\eps}), \\ \nonumber
&&\   \forall u,v\in B(t,h_{\ref{CR2}}),\  u-v=\dl>0, \ t\in S_l^X.
\eq
\end{lemma}
The proof of Lemma \ref{BR4} is given in Section \ref{Sec-BR4}.

\section{Lower Bound for the Spectrum of Singularities and Proof of \\ Theorem \ref{point1}} \label{Aplications}
In this section we obtain the lower bound on the spectrum of singularities for the process $M$ defined in Theorem \ref{Theorem-Spect}.
Later in this section we prove Theorem \ref{point1}. \\ \\
The following lemma gives an upper bound to the H\"{o}lder exponent for a more general class of functions.
This lemma corresponds to Lemma 1 in \cite{jaffard2}.
\begin{lemma} \label{lemma-reg}
Let $t_0\in\re$, $d\in(0,1)$ and $F\in SR^2_d(0+)$. Let $\{r_n\}_{n=1,2,...}$ be a sequence of points converging to $t_0$ such that for each point $r_n$, there exists $h_{\ref{lemma-reg}}=h_{\ref{lemma-reg}}(r_n)$ and $s_n\in (0,1)$ such that
\bq \label{reg1}
|g(r_n+\dl)-g(r_n)|\geq s_nF(r_n+\dl,r_n), \ \forall \dl\in(0,h_{\ref{lemma-reg}}).
\eq
Assume that
\bq \label{reg11}
    l:=\liminf_{n\rr\infty} \frac{\mathrm{log}s_n}{\log|r_n-t_0|}<\infty.
\eq
Then
   \bq \label{ddd}
   \bar{h}^d_g(t_{0})\leq l+d.
    \eq
\end{lemma}
\begin{proof}
The proof of Lemma \ref{lemma-reg} follows the same lines as the proof of Lemma 1 in \cite{jaffard2}. \\\\
Let $P_{t_0}$ be an arbitrary polynomial of degree at most $\lfloor l+d\rfloor$ and let $\eps>0$ be arbitrary small.
Recall $d\in(0,1)$, then from Lemma \ref{lemma-aux1}(b), for every $n\in \mathds{N}$ there exists $h_n>0$ such that
\bq \label{epl1}
|P_{t_0}(r_n)-P_{t_0}(z)|&\leq& C(t_0)|r_n-z| \nonumber \\
&\leq&\frac{s_n}{2}|r_n-z|^{d+\eps} \nonumber \\
&\leq&\frac{s_n}{2}F(z,r_n)
, \ \forall z\in (r_n,r_n+h_n).
\eq
From (\ref{epl1}) and (\ref{reg1}) we get that there exists $N_1\in\mathds{N}$, such that for each $n>N_1$ we can choose $r'_n >r_n$ which satisfies
\bq \label{reg02}
|P_{t_0}(r_n)-P_{t_0}(r'_n)|
&\leq&\frac{s_n}{2}F(r'_n,r_n),
\eq
\bq \label{reg01}
|g(r_n)-g(r'_n)|\geq s_nF(r'_n,r_n),
\eq
and
\bq \label{reg00}
|r_n-r'_n|\leq \frac{1}{2}|r_n-t_0|.
\eq
From (\ref{reg02}) and (\ref{reg01}) we have
\bq \label{reg03}
|g(r_n)-P_{t_0}(r_n)-(g(r'_n)-P_{t_0}(r'_n))|&\geq& \frac{s_n}{2}F(r_n',r'_n).
\eq
Finally from (\ref{reg03}), (\ref{reg11}) and (\ref{reg00}), we get that there exists $N_2\geq N_1$, such that for every $n\geq N_2$
\bq \label{reg04}
|g(r_n)-P_{t_0}(r_n)-(g(r'_n)-P_{t_0}(r'_n))| &\geq& \frac{s_n}{2}F(r'_n,r_n) \nonumber \\
&\geq& \frac{1}{4}F(r'_n,r_n)|r_n-t_0|^{l+\eps} \nonumber \\
&\geq& C(t_0)F(r'_n,r_n)(|r_n-t_0|^{l+\eps}+|r'_n-t_0|^{l+\eps}).
\eq
Since $\eps$ and $h=|r_n-t_0|$ are arbitrarily small and $|r_n-r'_n|<h/2$, (\ref{ddd}) follows immediately from (\ref{reg04}) and (\ref{Def-F-Holderexponent}).
\end{proof}
\\\\
Before we start our proof of the lower bound on the Hausdorff dimension in Theorem \ref{Theorem-Spect}, we need to introduce some additional notations and make some assumptions, which are taken from \cite{jaffard}.
Let $\{X(t)\}_{t \geq 0 }$ be a L\'{e}vy process considered in Theorem \ref{Theorem-Spect},  without a Brownian component, with a L\'{e}vy measure $\pi(dx)$. We assume without a loss of generality that $\pi(\mathds{R}\setminus [-1/2,1/2])=0$.
We can do so since if $X$ has a finite number of jumps of absolute value greater than $1/2$, it has no effect on the spectrum of singularities. Up to a linear term $X(t)$ can be constructed as a superposition of independent compensated compound Poisson processes $X^j(t)$ with jump sizes
 \bd
 \Gamma_j=\{x:2^{-j-1}<|x|\leq 2^{-j}\}.
 \ed
Let $Y^j(t)$ be a compound Poisson with a Levy measure
 \bd
 \pi_j(dx)=1_{\Gamma_j}(x)\pi(dx).
 \ed
Suppose $Y^1(t),Y^2(t),\ldots$ are independent processes and let
\bd
X^j(t) = Y^j(t) - t \int_{\mathds{R}} x \pi_j(dx).
\ed
Then $X^j(t)$ are independent processes and we can define $X$ as
\bd
X(t) = \sum_{j=1}^{\infty} X^j(t) + at,
\ed
for some $a\in \mathds{R}$. We assume that $a=0$, since this again does not affect the spectrum of singularities.
The intensity of $X^j(t)$ equals 
\be \label{cj}
C_j=\int_{2^{-j-1}< |x|\leq 2^{-j}}\pi(dx).
\ee
Denote by $S$ the set of jump times of $X$ on $t\in[0,1]$.
Let $F_j$ be a set of times of the jumps of $X^j(t)$ on $t\in[0,1]$ and let $\dl>0$.
Denote by $A_\dl^j$ the union of closed intervals of length $2\cdot 2^{-\dl\cdot j}$ centered at the points of $F_j$. \\\\
Denote by $E_\dl$ the random set
\be \label{edl}
E_\dl=\lim_{j \rr \infty} \sup A_\dl^j.
\ee
Recall that $S^{X}_{1/\dl}$ is the set of points $t\in[0,1]$ such that the H\"{o}lder exponent $h_X(t)$ of $X_t$ equals to $1/\dl$.
\paragraph{Convention:} Throughout this Section the results are stated in $P$-a.s. sense but without explicit using symbols "$P$-a.s.". This is done in order to improve the readability of the section.  \\\\
Now we prove an essential lemma that is needed for the proof of the lower bound on the spectrum of singularities in Theorem \ref{Theorem-Spect}.
\begin{lemma} \label{lemma-dimE}
Let $t\in E_\dl$, then
\bq \label{dimE}
   \bar{h}^d_M(t)\leq l+d.
\eq
\end{lemma}
\begin{proof}
Let $\eps>0$ be arbitrary small. By the construction of $E_\dl$ we can extract a sequence of jump points $\{r_n\}_{n\geq 0}$ of the process $X$, converging to $t$, such that jump size $s_n$ of the process $X$ at $r_n$ satisfies $s_n\in[ 2^{-n-1},2^{-n}]$, and $|r_n-t|\leq 2^{-\dl n}$. From Theorem \ref{thm1}(a) we get that for each $r_n$, there exists $h_n>0$, such that
\begin{displaymath}
 |M(r_n+u)-M(r_n)|\geq \frac{1}{2}|s_n|F(r_n+u,r_n) , \ \ \forall u\in (0,h_n), \ n \in \mathds{N}.
\end{displaymath}
Note that
\bq \label{t12}
\liminf_{n\rr\infty} \frac{\mathrm{log}s_n}{\log|r_n-t_0|}\leq 1/\dl.
\eq
From (\ref{t12}) and Lemma \ref{lemma-reg}, (\ref{dimE}) follows immediately.
\end{proof}
\\\\
In the following proposition we derive the lower bound of the Hausdorff dimension of the sets $Q_{1/\dl,d}^M$.
Recall that $S$ is the set of jump times of $X$.
\begin{proposition} \label{Prop-Q-dim}
For all $\dl> \beta $ we have
\begin{itemize}
\item[{\bf (a)}]
\bn
S^{X}_{1/\dl} \setminus S \subset Q_{1/\dl,d}^M,
\en
\item[{\bf (b)}]
\bq \label{Res-Lemma-E}
\frac{\beta}{\dl}\leq\dim(Q_{1/\dl,d}^M). \nonumber
\eq
\end{itemize}
\end{proposition}
\begin{proof}
In Proposition 1 in \cite{jaffard} it was proved that
\begin{eqnarray}  \label{funix}
S^X_{1/\dl} &\subset& \bigcap_{0<\alpha < \dl } E_{\alpha}, \ \ \forall \dl \in (0,\infty),
\end{eqnarray}
\be \label{funix1}
S^X_{1/\dl} = \bigg(\bigcap_{\alpha>0 } E_{\alpha}\bigg)\cup S , \ \  \dl =\infty,
\ee
without the use of assumption (\ref{sumcond}).
From (\ref{funix}) and (\ref{funix1}) we get
\be \label{funi221}
S^X_{1/\dl} \setminus S  \subset \bigcap_{0<\alpha < \dl } E_{\alpha} ,  \ \ \forall \dl \in (0,\infty].
\ee
Notice that $E_{\dl}$ is decreasing in $\dl$. From (\ref{funi221}) we have
\be \label{n33331}
S^X_{1/\dl} \setminus S  \subset E_{\dl-\eps}, \ \ \forall \eps \in(0,\dl).
\ee
Let $t_0\in S^X_{1/\dl}\setminus S$. From (\ref{n33331}) and Lemma \ref{lemma-dimE} we have
\bq \label{n41}
   \bar{h}^d_M(t_0)\leq \frac{1}{\dl-\eps}+d, \ \  \forall \eps\in(0,\dl).
\eq
From Theorem \ref{Corollary- point1} we have
\bq \label{n42}
\bar{h}^d_M(t_0) \geq \frac{1}{\dl}+d.
\eq
Combine (\ref{n41}) and (\ref{n42}) to get that $t_0\in Q_{1/\dl,d}^M$ and therefore
\bq \label{n43}
S^{X}_{1/\dl} \setminus S \subset Q_{1/\dl,d}^M.
\eq
Balan\c{c}a in Section 2.1 of \cite{Balanca} proved that
\bq \label{funi222}
\dim (S^{X}_{1/\dl} \setminus S) =\beta/\dl.
\eq
From (\ref{n43}) and (\ref{funi222}), and basic properties of the Hausdorff dimension, (b) follows.
\end{proof}
\paragraph{Proof of Theorem \ref{point1}}
Let $l\in (0,\infty)$ and $t\in S_l^X$. Note that for every $\eps,r>0$ there exists $C(t)>0$ such that
\bq \label{n124}
|u-t|^{r+\eps}+|v-t|^{r+\eps}&\geq& \frac{1}{4}(|u-t|^{\eps}+|v-t|^{\eps})(|u-t|^{r}+|v-t|^{r}) \nonumber \\
&\geq& C(t)|u-v|^{\eps}(|u-t|^{r}+|v-t|^{r}),\ \forall u,v\in [0,1].
\eq
Since $l\in (0,\infty)$ we get from (\ref{funix1}) that $t\not\in S$. Fix an arbitrarily small $\eps>0$. From Proposition \ref{Prop-Q-dim}(a), we get that for every $h>0$ and a polynomial $P_t$ of degree at most $\lfloor l+d \rfloor$, there exist $u,v\in B(t,h)$ such that
\bq \label{n123}
|M(u)-P_t(u)-(M(v)-P_t(v))|\geq F(u,v)(|u-t|^{l+\eps}+|v-t|^{l+\eps}).
\eq
Apply Lemma \ref{lemma-aux1}(b) and then (\ref{n124}) to get from (\ref{n123})
\bq \label{n125}
|M(u)-P_t(u)-(M(v)-P_t(v))|&\geq& C(t)|u-v|^{d+\eps}(|u-t|^{l+\eps}+|v-t|^{l+\eps}) \nonumber \\
&\geq & C(t)|u-v|^{d+2\eps}(|u-t|^{l}+|v-t|^{l}).
\eq
Since $\eps$ and  $h$ are arbitrarily small, we get from (\ref{def-frontier}) and (\ref{n125}) that
\bq
\sigma_{M,t}(-l) \leq d.
\eq
From Proposition \ref{propoint1}, Lemma \ref{lemma-aux1}(a) and (\ref{n124}) we get that for any $\eps\in(0,l)$, there exist $C_{\ref{propl}}=C_{\ref{propl}}(\omega, t)$, $h_{\ref{propl}}=h_{\ref{propl}}(\omega, t)$ and a polynomial $\tilde P_t$ of degree at most $\lfloor l+d \rfloor$ such that
\bq \label{propl}
&&|M(u)-\tilde P_t(u)-(M(v)- \tilde  P_t(v))| \nonumber \\
&&\leq C_{\ref{propoint11}}(t)F(u,v)(|u-t|^{l-\eps}+|v-t|^{l-\eps}) \nonumber \\
&&\leq C_{\ref{propoint11}}(t)(u-v)^{d-\eps}(|u-t|^{l-\eps}+|v-t|^{l-\eps}) \nonumber \\
&&\leq C(t)(u-v)^{d-2\eps}(|u-t|^{l}+|v-t|^{l})
, \ \  \forall u,v\in B(t,h_{\ref{propoint11}}), \ u>v>0 , \ \ t\in S_l^X, \ \ P-\rm{a.s.}
\eq
From (\ref{propl}) and (\ref{def-frontier}) we have
\bq
\sigma_{M,t}(-l) \geq d,
\eq
and we are done. \qed
\section{Upper Bound for the Spectrum of Singularities and Proof of \\ Theorem \ref{Theorem-Spect}} \label{Aplications2}
In this section we obtain an upper bound for the spectrum of singularities of the process $M$ defined in Theorem \ref{Theorem-Spect}. The proof is based on the results of Theorem \ref{Corollary- point1} and of Theorem 1 in \cite{jaffard}. Later in this section we prove Theorem \ref{Theorem-Spect}. We use here the same notation as in Section \ref{Multi}.
We also set
\be \label{Rdldef}
\tilde{S}^X_{\dl^{'}}=
\left \{ \begin{array}{ll}
\bigcup_{\dl^{'} \leq \dl} S^{X}_{1/\dl}, \ \ \textrm{    if   } \dl^{'} \in (0,\infty), & \\ \\
S^{X}_{\infty}, \  \  \textrm{  if    } \dl^{'}=\infty.  \\
\end{array} \right.
\ee
The following lemmas are crucial for the derivation of the upper bound of the Hausdorff dimension of the sets $Q_{1/\dl,d}^M $ in (\ref{Sets-QM}).
Recall that $S$ is the set of jump times of $X$.
\begin{lemma} \label{rdltag}
Let $\dl^{'} \in (0,\infty]$.
Then
\bd
\tilde{S}^X_{\dl^{'}} \setminus S \subset \bigcap_{0<\alpha < \dl^{'}} E_{\alpha}.
\ed
\end{lemma}
\begin{proof}
Let $\dl'\in(0,\infty]$. Note that $E_{\dl}$ is decreasing in $\dl$, then by (\ref{funi221}) we get
\be \label{funi2}
\bigg(\bigcup_{\dl^{'} \leq \dl} S^{X}_{1/\dl}\bigg) \setminus S  \subset \bigcap_{0<\alpha < \dl^{'} } E_{\alpha} ,  \ \ \forall \dl^{'} \in (0,\infty].
\ee
By (\ref{Rdldef}) and (\ref{funi2}) the result follows.
\end{proof}
\begin{lemma} \label{n8}
Let $M$ be defined as in Theorem \ref{Theorem-Spect}. Define
\be \label{tqdlm}
\tilde{Q}_{1/\dl,d}^M = \{t \in [0,1]: \bar{h}^d_M(t) \leq 1/\dl+ d\}\setminus S.
\ee
 Then
\bd
\tilde{Q}_{1/\dl,d}^M \subset \bigcap_{0<\alpha < \dl} E_{\alpha}, \ \ \forall 1/\dl \in [0,1/\beta).
\ed
\end{lemma}
\begin{proof}
Let $1/\dl\in(0,1 / \beta)$ and suppose that $t_0\in [0,1]$ and $t_0 \not \in  \bigcap_{0<\alpha < \dl} E_{\alpha}$.
Our goal is to show that
\begin{eqnarray*}
t_0 &\in&  \big(\tilde{Q}_{1/\dl,d}^M\big)^c  \\ \\
&=&  \{t \in [0,1]:\bar{h}^d_M(t) \leq 1/\dl+ d\}^c \cup S.
\end{eqnarray*}
This will prove the lemma.
\medskip
If $t_0\in S$ the proof is finished. Hence from now on we assume that $t_0\not \in S$ and our aim is to show that
$t_0 \not \in   \{t\in[0,1] :\bar{h}^d_M(t) \leq 1/\dl+ d\}$. \\\\
First we show that $t_0\in\bigcup_{\dl'\leq \dl}S^X_{1/\dl'}$.
Recall that $E_{\dl}$ is decreasing in $\dl$ and therefore if $t_0 \not \in  \bigcap_{0<\alpha < \dl} E_{\alpha}$ then necessarily $t_0 \not \in  E_{\dl}$. By Proposition 1 in \cite{Balanca}, since $t_0 \not \in  E_{\dl}$, $t_0$ is not a jump point, and $\dl>\beta$, then we have
\be \label{dlbig}
t_0\in\bigcup_{\dl'\leq \dl}S^X_{1/\dl'}.
\ee
Now we show that $t_0 \not \in \{t \in [0,1]:\bar{h}^d_M(t) \leq 1/\dl+ d\}$.
By Lemma \ref{rdltag}, we have
\be \label{n1}
\tilde{S}^X_{1/\dl} \setminus S \subset \bigcap_{0<\alpha < \dl} E_{\alpha}.
\ee
Therefore, for $t_0 \not \in  \bigcap_{0<\alpha < \dl} E_{\alpha}$, we have $t_0 \in  (\tilde{S}^X
_{1/\dl})^c \cup S$.
Since we consider now only $t_0\not \in S$ and we get from (\ref{dlbig}) that $t_0\in\bigcup_{\dl'\leq \dl}S^X_{1/\dl'}$,
this implies that
\be \label{n9}
h_X(t_0) > 1/\dl.
\ee
From Theorem \ref{Corollary- point1}, and (\ref{n9}) we get
\bd
\bar{h}^d_M(t) > 1/\dl+ d.
\ed
Therefore $t_0 \not \in \tilde{Q}_{1/\dl+d}^M$, and the result follows.
\end{proof}
\paragraph{Proof of Theorem \ref{Theorem-Spect}:}
In the end of Section 1 of \cite{jaffard} Jaffard showed, without using assumption (\ref{sumcond}), that for every $\dl>\beta$.
\be \label{n121}
\dim (E_{\dl}) \leq \beta/\dl.
\ee
Since $E_{\dl}$ is decreasing in $\dl$, we have
\be \label{n331}
\bigcap_{0<\alpha<\dl}E_{\alpha} \subset E_{\dl-\eps}, \ \ \forall \eps \in (0,\dl).
\ee
From (\ref{n121}) and (\ref{n331}) we get
\be \label{n332}
\dim \big(\bigcap_{0<\alpha<\dl}E_{\alpha}\big) \leq \frac{\beta}{\dl-\eps} , \ \ \forall \eps \in (0,\dl-\beta).
\ee
By (\ref{n332}) for every $\dl>\beta$ we have
\be \label{n333}
\dim \big(\bigcap_{0<\alpha<\dl}E_{\alpha}\big) \leq \frac{\beta}{\dl}.
\ee
By Lemma \ref{n8}
\be \label{n3}
\tilde{Q}_{1/\dl,d}^M \subset \bigcap_{0<\alpha < \dl} E_{\alpha}, \ \ \forall 1/\dl \in [0,1/\beta).
\ee
Since the set of jumps of the L\'{e}vy process $X$ is a countable set, it has no effect on the Hausdorff dimension.
Then by (\ref{n333}) and (\ref{n3}) we get
\begin{eqnarray}  \label{n5}
\dim \{t\in[0,1]:\bar h^d_M(t) \leq 1/\dl+d\}  &=&  \dim (\tilde{Q}_{1/\dl,d}^X)  \\  \nonumber
&\leq&  \beta/\dl ,\ \ \forall 1/\dl \in [0,1/\beta). \nonumber
\end{eqnarray}
From (\ref{n5}) and Proposition \ref{Prop-Q-dim}(b), Theorem~\ref{Theorem-Spect} follows.
\qed

\section{ Proof of Lemma \ref{BR4}}
This section is devoted to the proof of Lemma \ref{BR4}. The proof of Lemma \ref{BR4} follows immediately from the following sequence of lemmas.
Recall that our goal is to show that there exists a polynomial $P^2_t$ of degree of most $\lfloor l+d \rfloor$ and constants $C(\omega,t)>0$, $h(\omega,t)>0$ such that for every $\eps\in(0,l)$,
\bq \label{dff1}
\bigg|\int_{0}^{v}(f(u,r)-f(v,r))X(r)dr-P^2_t(u)+P^2_t(v)\bigg|  &\leq& C(t)F(u,v)(|u-t|^{l-\eps}+|v-t|^{l-\eps}), \nonumber \\
&&  \  \forall u,v\in B(t,h), u>v,  t\in S_l^X.
\eq
We prove (\ref{dff1}) for the case where $t<v$. The proof for the case where $t\geq v$ follows the same lines.
\begin{lemma} \label{BR28}
Under the same assumption as in Lemma \ref{BR4},
for any $\eps \in (0,l)$ there exists a polynomial $P^3_t$ of degree at most $\lfloor l+d \rfloor$, and constants $C_{\ref{CR28}}=C_{\ref{CR28}}(\omega,t)>0$ and $h_{\ref{CR28}}=h_{\ref{CR28}}(\omega,t)>0$ such that for every $h\in(0,t)$,
\bq \label{CR28}
&&\bigg|\int_{0}^{t-h}(f(u,r)-f(v,r))[X(r)-P_t(r)dr-P^3_t(u)+P^3_t(v)]\bigg| \nonumber \\ &&\leq C_{\ref{CR28}}F(u,v)(|u-t|^{l-\eps}+|v-t|^{l-\eps}), \ \  \forall u,v\in B(t,h_{\ref{CR28}}),\ u>v, \ \ t\in S_l^X.
\eq
\end{lemma}
\begin{proof}
Assume that $t\in S_l^X$ is fixed.
By Definition \ref{smt.var2.0}, for every $h\in (0,t)$ we have
\bq \label{rf1}
\sup_{ 0\leq t\leq 1 }\sup_{ 0\leq r\leq t-h }|f^{(k,0)}(t,r)|<\infty, \ \forall k=0,1,..,\lfloor l+d \rfloor+1.
\eq
Since $[0,t-h]\subset [0,1]$ and $r\in[t-h,t]$, we can apply Lemma \ref{bound} and the definition of $S_l^X$ to see that for every $\varepsilon \in (0,l)$, there exists $C_{\ref{BoundX12}}(t,\omega)$ such that
\bq \label{BoundX12}
|X(r)-P_t(r)|&\leq& C_{\ref{BoundX12}}(t,\omega)|r-t|^{l-\varepsilon},
\ \forall r\in[0,t-h], \ \ t\in S_l^X.
\eq
Let
\bq \label{rf0}
P^3_t(z):=\sum_{k=0}^{\lfloor l +d \rfloor}\frac{(z-t)^k}{k!} \int_{0}^{t-h}f^{(k,0)}(t,r)[X(r)-P_t(r)]dr.
\eq
From (\ref{rf0}), the Taylor Theorem, (\ref{BoundX12}), (\ref{rf1}) and Lemma \ref{lemma-aux1}(b) we get
\bq \label{rf122}
&&\bigg|\int_{0}^{t-h}(f(u,r)-f(v,r))[X(r)-P_t(r)]dr-P^3_t(u)+P^3_t(v)\bigg|  \nonumber \\
&& =\frac{1}{(\lfloor l +d \rfloor+1)!}\bigg|\int_{0}^{t-h}\bigg[\int_{t}^{u}f^{(\lfloor l +d \rfloor+1,0)}(z,r)(u-z)^{\lfloor l +d \rfloor+1}dz \nonumber \\
&&-\int_{t}^{v}f^{(\lfloor l +d \rfloor+1,0)}(z,r)(v-z)^{\lfloor l +d \rfloor+1}dz\bigg][X(r)-P_t(r)]dr\bigg| \nonumber \\
&& =C\bigg|\int_{0}^{t-h}\bigg[\int_{t}^{v}f^{(\lfloor l +d \rfloor+1,0)}(z,r)[(u-z)^{\lfloor l +d \rfloor+1}-(v-z)^{\lfloor l +d \rfloor+1}]dz \nonumber \\
&&+\int_{v}^{u}f^{(\lfloor l +d \rfloor+1,0)}(z,r)(u-z)^{\lfloor l +d \rfloor+1}dz\bigg][X(r)-P_t(r)]dr\bigg| \nonumber \\
&& \leq C(t)\bigg|\int_{0}^{t-h}\bigg|\int_{t}^{v}f^{\lfloor l +d \rfloor+1}(z,r)[(u-z)^{\lfloor l +d \rfloor+1}-(v-z)^{\lfloor l +d \rfloor+1}]dz \nonumber \\
&&+\int_{v}^{u}f^{(\lfloor l +d \rfloor+1,0)}(z,r)(u-z)^{\lfloor l +d \rfloor+1}dz\bigg|(t-r)^l dr \nonumber \\
&& \leq C(t)\int_{0}^{t-h}\bigg[\int_{t}^{v}[(u-z)^{\lfloor l +d \rfloor+1}-(v-z)^{\lfloor l +d \rfloor+1}]dz
+\int_{v}^{u}(u-z)^{\lfloor l +d \rfloor+1}dz\bigg]dr \nonumber \\
&& \leq C(t)(u-v)(|u-t|^{\lfloor l +d \rfloor+1-\eps}+|v-t|^{\lfloor l +d \rfloor+1-\eps})  \nonumber \\
&& \leq C(t)F(u-v)(|u-t|^{l-2\eps}+|v-t|^{l-2\eps}),  \nonumber
\eq
where $C(t)=C(t,\omega)$.
\end{proof}
\begin{lemma} \label{lem-conv}
Let $l>0$, $d\in(0,1)$, $F\in SR^{\lfloor d+l \rfloor +2}_d(0+)$ and fix $t\in (0,1]$. There exists $h\in(0,t)$ such that for every $\eps>0$,
\bq \label{conv22}
\lim_{\dl\dr 0}\dl^{2\eps}\int_{t-h}^{t-\dl} (t-r)^{l-\eps}\int_{t}^{t+h}|f^{(\lfloor l +d \rfloor+1,0)}(z,r)|dzdr <\infty, \ \forall t\in[0,1].
\eq
\end{lemma}
\begin{proof}
By Definition \ref{smt.var2.0}, there exists $h\in(0,1)$ such that,
\be \label{mono0}
|f^{(k,0)}(t,r)|>|f^{(k,0)}(t+h,r)|, \ \ \forall r \in [t-h,t),\  k=0,1,...,\lfloor l +d \rfloor,
\ee
for $k\in \mathds{N}$ which satisfies $k\leq \frac{1}{2}(\lfloor l +d \rfloor +1)$ we have
\be \label{mono1}
f^{(2k,0)}(z,r)<0, \ \ \forall z-r \in (0,2h],
\ee
and for $k\in \mathds{N}$ which satisfies $k\leq \frac{1}{2}\lfloor l +d \rfloor $ we have
\be \label{mono2}
f^{(2k+1,0)}(z,r)>0, \ \ \forall z-r \in (0,2h].
\ee
From (\ref{mono0}--\ref{mono2}) we have
\bq \label{dd1}
\int_{t}^{t+h}|f^{(\lfloor l +d \rfloor+1,0)}(z,r)|dz
&\leq& |f^{(\lfloor l +d \rfloor,0)}(t+h,r)-f^{(\lfloor l +d \rfloor,0)}(t,r)|\nonumber \\
&\leq& 2|f^{(\lfloor l +d\rfloor,0)}(t,r)|.
\eq
From (\ref{dd1}), (\ref{mono1}) and  (\ref{mono2}) we have
\bq \label{dd2}
\int_{t-h}^{t-\dl}(t-r)^{l-\eps}\int_{t}^{t+h}|f^{(\lfloor l +d \rfloor+1,0)}(z,r)|dzdr &\leq& 2\int_{t-h}^{t-\dl}(t-r)^{l-\eps}|f^{(\lfloor l +d \rfloor,0)}(t,r)|dr
\nonumber \\
&= & 2\bigg|\int_{t-h}^{t-\dl}(t-r)^{l-\eps}f^{(\lfloor l +d \rfloor,0)}(t,r)dr \bigg|.
\eq
Then (\ref{conv22}) follows from (\ref{dd2}) and Lemma \ref{lemma-aux1}(a).
\end{proof}
\begin{lemma} \label{BR25}
Under the same assumption as in Lemma \ref{BR4},
for any $\eps \in (0,l)$, there exists a polynomial $P^4_t$ of degree at most $\lfloor l+d \rfloor$ and constants $C_{\ref{CR25}}=C_{\ref{CR25}}(\omega,t)>0$, $h_{\ref{CR25}}=h_{\ref{CR25}}(\omega,t)>0$ and $h\in(0,t)$, such that,
\bq \label{CR25}
&&\bigg|\int_{t-h}^{t-\dl}(f(u,r)-f(v,r))[X(r)-P_t(r)]dr-P^4_t(u)+P^4_t(v)\bigg| \nonumber \\ &&\leq C_{\ref{CR25}}F(u,v)(|u-t|^{l-\eps}+|v-t|^{l-\eps}), \ \  \forall u,v\in (t,t+h_{\ref{CR25}}),\ u-v=\dl>0, \ \ t\in S_l^X.
\eq
\end{lemma}
\begin{proof}
Assume that $t\in S_l^X$ is fixed. Let $h$ be as in Lemma \ref{lem-conv} and $u,v\in (t,t+h)$. Since $[t-h,t]\subset [0,1]$ and $r\in[t-h,t]$, we can apply Lemma \ref{bound} and the definition of $S_l^X$ to see that for every $\varepsilon \in (0,l)$, there exists $C_{\ref{BoundX11}}(t,\omega)$ such that
\bq \label{BoundX11}
|X(r)-P_t(r)|&\leq& C_{\ref{BoundX11}}(t,\omega)|r-t|^{l-\varepsilon},
\ \forall r\in[t-h,t], \ \ t\in S_l^X.
\eq
Let
\bq \label{rf00}
P^4_t(z):=\sum_{k=0}^{\lfloor l +d \rfloor}\frac{(z-t)^k}{k!} \int_{t-h}^{t-\dl}f^{(k,0)}(t,r)[X(r)-P_t(r)]dr.
\eq
From (\ref{rf00}), (\ref{BoundX11}) and the Taylor Theorem we get
\bq \label{rf01}
&&\bigg|\int_{t-h}^{t-\dl}(f(u,r)-f(v,r))[X(r)-P_t(r)]dr-P^4_t(u)+P^4_t(v)\bigg|  \nonumber \\
&& =\frac{1}{(\lfloor l +d \rfloor+1)!}\bigg|\int_{t-h}^{t-\dl}\bigg[\int_{t}^{u}f^{(\lfloor l +d \rfloor+1,0)}(z,r)(u-z)^{\lfloor l +d \rfloor+1}dz \nonumber \\
&&-\int_{t}^{v}f^{(\lfloor l +d \rfloor+1,0)}(z,r)(v-z)^{\lfloor l +d \rfloor+1}dz\bigg][X(r)-P_t(r)]dr\bigg| \nonumber \\
&& =C\bigg|\int_{t-h}^{t-\dl}\bigg[\int_{t}^{v}f^{(\lfloor l +d \rfloor+1,0)}(z,r)[(u-z)^{\lfloor l +d \rfloor+1}-(v-z)^{\lfloor l +d \rfloor+1}]dz \nonumber \\
&&+\int_{v}^{u}f^{(\lfloor l +d \rfloor+1,0)}(z,r)(u-z)^{\lfloor l +d \rfloor+1}dz\bigg][X(r)-P_t(r)]dr\bigg| \nonumber \\
&& \leq C\bigg|\int_{t-h}^{t-\dl}\bigg|\int_{t}^{v}f^{(\lfloor l +d \rfloor+1,0)}(z,r)[(u-z)^{\lfloor l +d \rfloor+1}-(v-z)^{\lfloor l +d \rfloor+1}]dz \nonumber \\
&&+\int_{v}^{u}f^{(\lfloor l +d \rfloor+1,0)}(z,r)(u-z)^{\lfloor l +d \rfloor+1}dz\bigg|(t-r)^{l-\eps} dr,
\eq
where $C=C(t,\omega)$.
Recall that $v\in (t,t+h)$ then from Lemma \ref{lem-conv} we get
\bq \label{h-1}
&&\int_{t-h}^{t-\dl}\bigg|\int_{t}^{v}f^{(\lfloor l +d \rfloor+1,0)}(z,r)[(u-z)^{\lfloor l +d \rfloor+1}-(v-z)^{\lfloor l +d \rfloor+1}]dz\bigg|(t-r)^{l-\eps} dr \nonumber \\
&&\leq C(t)(u-v)(u-t)^{\lfloor l +d \rfloor}\int_{t-h}^{t-\dl}\int_{t}^{v}|f^{(\lfloor l +d \rfloor+1,0)}(z,r)|dz(t-r)^{l-\eps} dr \nonumber \\
&&\leq C(t)(u-v)^{d+\eps}(u-t)^{\lfloor l +d \rfloor+1-d-\eps}\int_{t-h}^{t-\dl}\int_{t}^{v}|f^{(\lfloor l +d \rfloor+1,0)}(z,r)|dz(t-r)^{l-\eps} dr \nonumber \\
&&\leq C(t)F(u,v)(u-t)^{l-3\eps},
\eq
where the last inequality follows from Lemma \ref{lemma-aux1}(b) and our choice of $u-v=\dl$.
Recall that $u\in (t,t+h)$. Apply Lemma \ref{lem-conv} and Lemma \ref{lemma-aux1}(b) again to get
\bq \label{h-4}
&&\int_{t-h}^{t-\dl}\int_{v}^{u}|f^{(\lfloor l +d \rfloor+1,0)}(z,r)|(u-z)^{\lfloor l +d \rfloor+1}dz(t-r)^{l-\eps} dr  \nonumber \\
&&\leq
\dl^{\lfloor l +d \rfloor+1}\int_{t-h}^{t-\dl}\int_{t}^{v}|f^{(\lfloor l +d \rfloor+1,0)}(z,r)|(t-r)^{l-\eps} dr\nonumber \\
&&\leq C(t)(u-v)^{d+l-2\eps}\nonumber \\
&&\leq C(t)(u-v)^{d+\eps}(u-v)^{l-3\eps}\nonumber \\
&&\leq C(t)F(u,v)(u-v)^{l-3\eps}.
\eq
Apply (\ref{h-1}) and (\ref{h-4}) on (\ref{rf01}) to get (\ref{CR25}).
\end{proof}
\begin{lemma} \label{BR26}
Under the same assumption as in Lemma \ref{BR4}, for any $\eps \in (0,l)$, there exist constants $C_{\ref{CR26}}=C_{\ref{CR26}}(\omega,t)>0$ and $h_{\ref{CR26}}=h_{\ref{CR26}}(\omega,t)>0$ such that for every $\eps>0$,
\bq \label{CR26}
&&\bigg|\int_{t-\dl}^{v}(f(u,r)-f(v,r))[X(r)-P_t(r)]dz\bigg|,  \nonumber \\ &&\leq C_{\ref{CR26}}F(u,v)(|u-t|^{l-\eps}+|v-t|^{l-\eps}), \ \ \forall u,v\in (t,t+h_{\ref{CR25}}), \  u-v=\dl>0.
\eq
\end{lemma}
\begin{proof}
Assume that $t\in S_l^X$ is fixed.
Since $r-t \in[0,1]$, we can apply Lemma \ref{bound} and the definition of $S_l^X$ to see that for every $\varepsilon \in (0,l)$, there exists $C_{\ref{BoundX1}}(t,\omega)$ such that
\bq \label{BoundX1}
|X(r)-P_t(r)|\leq C_{\ref{BoundX1}}(t,\omega)|r-t|^{l-\varepsilon},
\ \forall r\in[0,1],  \ t\in S_l^X.
\eq
From (\ref{BoundX1}) we have
\bq \label{derivativer1}
&&\bigg|\int_{t-\dl}^{v}(f(u,r)-f(v,r))[X(r)-P_t(r)]dr\bigg|    \leq |v-t+\dl|^{l-\varepsilon}\int_{t}^{v}|f(v+\dl,r)-f(v,r)|dr.  \\ \nonumber
\eq
By Definition \ref{smt.var2.0}, there exists $\bar h>0$ such that for every $v,r,\dl \in [0,1]$ satisfying $v+\dl-r \in (0,\bar h)$, we have
\be \label{mono}
f(v+\dl,r)>f(v,r).
\ee
Recall that $u=v+\dl$. From (\ref{derivativer1}) and (\ref{mono}) and direct integration we have
\bq \label{derivativer2}
&& \bigg|\int_{t-\dl}^{v}(f(u,r)-f(v,r))[X(r)-P_t(r)]dr\bigg| \nonumber  \\
&& \leq |v+\dl-t|^{l-\varepsilon}[F(v+\dl,v)-F(v,v)-F(v+\dl,t-\dl)+F(v,t-\dl)],\ \forall v+2\dl-t\in (0,\bar h).
\eq
Recall that $F(v,v)=0$. We get from (\ref{derivativer2}) and Definition \ref{smt.var2.0}, that there exists $h\in(0,\bar h)$ such that
\begin{eqnarray} \label{derivativer222}
\frac{1}{F(v+\dl,v)}\bigg|\int_{t}^{v}(f(u,r)-f(v,r))[X(r)-P_t(r)]dr\bigg|\leq C|v-t|^{l-\varepsilon}, \ \  \forall u,v\in B(t, h), \ u-v=\dl>0.
\end{eqnarray}
and we are done.

\end{proof}

\paragraph{Proof of Lemma \ref{BR4}} \label{Sec-BR4}
Take $P^2_t=P^3_t+P^4_t$ and the proof of Lemma follows immediately from Lemmas \ref{BR28}, \ref{BR25} and \ref{BR26}.

\section{Acknowledgments}
This paper was written during my Ph.D. studies under
the supervision of Professor L. Mytnik. I am grateful to him for his guidance
and numerous helpful conversations during the preparation of this work.
I also thank an anonymous referee for the careful reading of the manuscript,
and for a number of useful comments and suggestions that improved the exposition.

\bibliographystyle{plain}
\printindex


\end{document}